\documentclass[11pt]{amsart}

\usepackage{amscd,amssymb,amsopn,amsmath,amsthm,mathrsfs,graphics,amsfonts,enumerate,verbatim,calc
}
\usepackage{bbm}
\usepackage[all,cmtip]{xy}
\usepackage[all]{xy}
\usepackage{tikz}
\usepackage{caption}
\usepackage{lscape}
\usetikzlibrary{matrix,arrows,decorations.pathmorphing}

\usepackage{scalerel}
\usepackage{stackengine,wasysym}
\usetikzlibrary {positioning}
\usetikzlibrary{patterns}
\usetikzlibrary{calc}
\definecolor {processblue}{cmyk}{0.96,0,0,0}

\usepackage{subfigure}

\usepackage{comment} 

\usepackage{ dsfont }

\usepackage{color}

\usepackage[OT2,OT1]{fontenc}
\newcommand\cyr{%
\renewcommand\rmdefault{wncyr}%
\renewcommand\sfdefault{wncyss}%
\renewcommand\encodingdefault{OT2}%
\normalfont
\selectfont}
\DeclareTextFontCommand{\textcyr}{\cyr}

\usepackage{amssymb,amsmath}

\DeclareFontFamily{OT1}{rsfs}{}
\DeclareFontShape{OT1}{rsfs}{n}{it}{<-> rsfs10}{}
\DeclareMathAlphabet{\mathscr}{OT1}{rsfs}{n}{it}

\numberwithin{equation}{section}
\newtheorem{theorem}{Theorem}[section]
\newtheorem{lem}[theorem]{Lemma}
\newtheorem{cor}[theorem]{Corollary}

\newtheorem{question}{Question}

\theoremstyle{definition}
\newtheorem{prop}[theorem]{Proposition}
\newtheorem{defn}[theorem]{Definition}
\theoremstyle{remark}
\newtheorem{remark}[theorem]{Remark}

\newtheorem{example}[theorem]{Example}

\newcommand{\N}{N}

\DeclareMathOperator{\Hilb}{Hilb}

\renewcommand{\tilde}{\widetilde}

\newcommand{\lk}{\operatorname{lk}}

\newcommand{\Spec}{\operatorname{Spec}}

\newcommand{\reg}{\operatorname{reg}}

\newcommand{\pd}{\operatorname{pd}}

\newcommand{\astar}{\operatorname{ast}}

\newcommand{\depth}{\operatorname{depth}}

\newcommand{\str}{\ensuremath{\operatorname{st}}}
\newcommand{\ho}{\tilde H}

\newcommand{\m}{\mathfrak{m}}

\DeclareMathOperator{\lcm}{lcm}

\DeclareMathOperator{\sd}{sd} 

\newcommand{\set}[1]{\{ #1 \}}

\usepackage{xstring}

\newcommand{\drawsimplex}[4]{
	\IfEqCase{#1}{%
		{nw}{\draw [ultra thick, draw=black, draw opacity=1, pattern=north west lines, pattern color=blue](#2) \foreach \i in #3{ -- (\i)} -- cycle;}
		{ne}{\draw [ultra thick, draw=black, draw opacity=1, pattern=north east lines, pattern color=blue](#2) \foreach \i in #3{ -- (\i)} -- cycle;}
		{v}{\draw [ultra thick, draw=black, draw opacity=1, pattern=vertical lines, pattern color=blue] (#2) \foreach \i in #3{ -- (\i)} -- cycle;}
 		{h}{\draw [ultra thick, draw=black, draw opacity=1, pattern=horizontal lines, pattern color=blue] (#2) \foreach \i in #3{ -- (\i)} -- cycle;}
		{d}{\draw [ultra thick, draw=black, draw opacity=1, pattern=dots, pattern color=blue] (#2) \foreach \i in #3{ -- (\i)} -- cycle;}
	}[\PackageError{drawsimplex}{Done messed up with #1}{}];

\draw [ultra thick, draw=black, draw opacity=1] \foreach \i in #4 \foreach \j in #4{(\i) -- (\j) };

}

\newcommand{\drawintro}[1]{
\begin{tikzpicture}[scale=#1]
  \node[coordinate] at (0,0) (D){};
  \node at (D) [right] {$\mathbf{D}$};
  \node[coordinate] at (-1.3,0) (E){};
  \node at (E) [left] {$\mathbf{E}$};
  \node[coordinate] at (-0.3,2) (B){};
  \node at (B) [above] {$\mathbf{B}$};
  \node[coordinate] at (-0.3,-2) (G){};
  \node at (G) [below] {$\mathbf{G}$};
  \node[coordinate] at (1.3,-1.3) (H){};
  \node at (H) [below] {$\mathbf{H}$};
  \node[coordinate] at (1.3,1.3) (A){};
  \node at (A) [above] {$\mathbf{A}$};
  \node[coordinate] at (0.3,-1.2) (F){};
  \node at (F) [below] {$\mathbf{F}$};
  \node[coordinate] at (0.3,1.2) (C){};
  \node at (C) [above] {$\mathbf{C}$};
\drawsimplex{v}{D}{{H,G}}{{D,G,H,F}}
  \node at ($.25*(D)+.25*(G)+.25*(H)+.25*(F)$) [above right] {$\mathbf{F_1}$};
\drawsimplex{ne}{A}{{D,B}}{{A,C,D,B}}
  \node at ($.25*(A)+.25*(C)+.25*(D)+.25*(B)$) [below right] {$\mathbf{F_2}$};
\drawsimplex{nw}{B}{{C,D,E}}{{D,C,B,E}}
  \node at ($.25*(D)+.25*(C)+.25*(B)+.25*(E)$) [left] {$\mathbf{F_3}$};
\drawsimplex{h}{D}{{F,G,E}}{{D,G,E,F}}
  \node at ($.25*(D)+.25*(G)+.25*(E)+.25*(F)$) [left] {$\mathbf{F_4}$};
\end{tikzpicture}
}

\newcommand{\drawintroNone}[1]{
\begin{tikzpicture}[scale=#1]
  \node[coordinate] at (.5,-1) (F1){};
  \node at (F1) [below right] {$\mathbf{F_1}$};
  \node[coordinate] at (.5,1) (F2){};
  \node at (F2) [above right] {$\mathbf{F_2}$};
  \node[coordinate] at (-.5,.8) (F3){};
  \node at (F3) [above left] {$\mathbf{F_3}$};
  \node[coordinate] at (-.5,-.8) (F4){};
  \node at (F4) [below left] {$\mathbf{F_4}$};
  \node[coordinate] at (-0,2) (B){};
  \node at (B) [above] {};
  \node[coordinate] at (-0,-2) (G){};
  \node at (G) [below] {};
\drawsimplex{v}{F1}{{F2,F3,F4}}{{F1,F2,F3,F4}}
\draw[black,thick,fill=black] (F1) circle [radius=2pt];
\draw[black,thick,fill=black] (F2) circle [radius=2pt];
\draw[black,thick,fill=black] (F3) circle [radius=2pt];
\draw[black,thick,fill=black] (F4) circle [radius=2pt];
\end{tikzpicture}
}

\newcommand{\drawintroNtwo}[1]{
\begin{tikzpicture}[scale=#1]
  \node[coordinate] at (.5,-1) (F1){};
  \node at (F1) [below right] {$\mathbf{F_1}$};
  \node[coordinate] at (.5,1) (F2){};
  \node at (F2) [above right] {$\mathbf{F_2}$};
  \node[coordinate] at (-.5,.8) (F3){};
  \node at (F3) [above left] {$\mathbf{F_3}$};
  \node[coordinate] at (-.5,-.8) (F4){};
  \node at (F4) [below left] {$\mathbf{F_4}$};
  \node[coordinate] at (-0,2) (B){};
  \node at (B) [above] {};
  \node[coordinate] at (-0,-2) (G){};
  \node at (G) [below] {};
\draw [ultra thick, draw=black]  (F1) -- (F4);
\draw [ultra thick, draw=black]  (F3) -- (F4);
\draw [ultra thick, draw=black]  (F3) -- (F2);
\draw[black,thick,fill=black] (F1) circle [radius=2pt];
\draw[black,thick,fill=black] (F2) circle [radius=2pt];
\draw[black,thick,fill=black] (F3) circle [radius=2pt];
\draw[black,thick,fill=black] (F4) circle [radius=2pt];
\end{tikzpicture}
}

\newcommand{\drawintroNthree}[1]{
\begin{tikzpicture}[scale=#1]
  \node[coordinate] at (.5,-1) (F1){};
  \node at (F1) [below right] {$\mathbf{F_1}$};
  \node[coordinate] at (.5,1) (F2){};
  \node at (F2) [above right] {$\mathbf{F_2}$};
  \node[coordinate] at (-.5,.8) (F3){};
  \node at (F3) [above left] {$\mathbf{F_3}$};
  \node[coordinate] at (-.5,-.8) (F4){};
  \node at (F4) [below left] {$\mathbf{F_4}$};
  \node[coordinate] at (-0,2) (B){};
  \node at (B) [above] {};
  \node[coordinate] at (-0,-2) (G){};
  \node at (G) [below] {};
\draw [ultra thick, draw=black]  (F1) -- (F4);
\draw [ultra thick, draw=black]  (F3) -- (F2);
\draw[black,thick,fill=black] (F1) circle [radius=2pt];
\draw[black,thick,fill=black] (F2) circle [radius=2pt];
\draw[black,thick,fill=black] (F3) circle [radius=2pt];
\draw[black,thick,fill=black] (F4) circle [radius=2pt];
\end{tikzpicture}
}

\newcommand{\drawintroNfour}[1]{
\begin{tikzpicture}[scale=#1]
  \node[coordinate] at (.5,-1) (F1){};
  \node at (F1) [below right] {$\mathbf{F_1}$};
  \node[coordinate] at (.5,1) (F2){};
  \node at (F2) [above right] {$\mathbf{F_2}$};
  \node[coordinate] at (-.5,.8) (F3){};
  \node at (F3) [above left] {$\mathbf{F_3}$};
  \node[coordinate] at (-.5,-.8) (F4){};
  \node at (F4) [below left] {$\mathbf{F_4}$};
  \node[coordinate] at (-0,2) (B){};
  \node at (B) [above] {};
  \node[coordinate] at (-0,-2) (G){};
  \node at (G) [below] {};
\draw[black,thick,fill=black] (F1) circle [radius=2pt];
\draw[black,thick,fill=black] (F2) circle [radius=2pt];
\draw[black,thick,fill=black] (F3) circle [radius=2pt];
\draw[black,thick,fill=black] (F4) circle [radius=2pt];
\end{tikzpicture}
}

\begin{document}
\title[Higher Nerves of Simplicial Complexes]{Higher Nerves of Simplicial Complexes}

\author[Dao]{Hailong Dao}
\email[Hailong Dao]{hdao@ku.edu}
\address{Department of Mathematics\\
University of Kansas\\
Lawrence, KS 66045-7523 USA}
\author[Doolittle]{Joseph Doolittle}
\email[Joseph Doolittle]{jdoolitt@ku.edu}
\address{Department of Mathematics\\
University of Kansas\\
Lawrence, KS 66045-7523 USA}
\author[Duna]{Ken Duna}
\email[Ken Duna]{kduna@ku.edu}
\address{Department of Mathematics\\
University of Kansas\\
Lawrence, KS 66045-7523 USA}
\author[Goeckner]{Bennet Goeckner}
\email[Bennet Goeckner]{goeckner@uw.edu}
\address{Department of Mathematics\\
University of Washington\\
Seattle, WA 98195-4350 USA}
\author[Holmes]{Brent Holmes}
\email[Brent Holmes]{brentjholmes323@gmail.com}
\address{Department of Mathematics\\
University of Kansas\\
Lawrence, KS 66045-7523 USA}
\author[Lyle]{Justin Lyle}
\email[Justin Lyle]{justin.lyle@ku.edu}

\date{\today}



\begin{abstract}
We investigate generalized notions of the nerve complex for the facets of a simplicial complex. We show that the homologies of these higher nerve complexes determine the depth of the Stanley-Reisner ring $k[\Delta]$ as well as the $f$-vector and $h$-vector of $\Delta$.  We present, as an application, a formula for computing regularity of monomial ideals.
\end{abstract}

\maketitle

\section{Introduction}

The Nerve complex has been an important object of study in algebraic combinatorics \cite{Bj03, Gr70, Bo48, Ba03, KM05, LS11, CJ15, PU16}.  We remind the reader of its definition:

Let $A = \{ A_1,A_2,\dots,A_r \}$ be a family of sets.

\begin{defn}
 Consider
\[\N(A) := \{ F \subseteq [r] \colon \cap_{i\in F}A_{i} \neq  \emptyset \}.\]
This simplicial complex is the \textit{Nerve Complex} of $A$.
\end{defn}
Of special interest is the case where $A$ is the set of facets of a simplicial complex $\Delta$; in this case, one sets $N(\Delta):=N(A)$.
We propose a natural extension of this notion.

\begin{defn}\label{highernervedef}
Let $A = \{ A_1,A_2,\dots,A_r \}$ be the set of facets of a simplicial complex $\Delta$.  Define
\[\N_i(\Delta) := \{ F \subseteq [r] \colon |\cap_{j\in F}A_{j}| \geq i \}.\]
We call this simplicial complex the \textit{$i^{th}$ Nerve Complex} of $\Delta$ and we refer to the $N_i(\Delta)$ and the \textit{higher Nerves Complexes} of $\Delta$.
\end{defn}
When $i = 1$, this definition recovers $N(\Delta)$.

The Nerve Theorem of Borsuk \cite{Bo48} gives that $\N(\Delta)$ and $\Delta$ have the same homologies. We now explain how the higher nerves relate to the original complex in a more subtle manner. Namely, their homologies determine important algebraic and combinatorial properties of $\Delta$. We summarize our main quantitative results below.

\begin{theorem}[Main Theorem]\label{LongConj}
Let $k$ be a field, let $\Delta$ be a simplicial complex of dimension $d-1$, and let $k[\Delta]$ be the associated Stanley-Reisner ring.  Let $\tilde{H}_i$ denote $i$th reduced simplicial homology with coefficients in $k$, and let $\chi$ denote Euler characteristic.  Then:
\begin{enumerate}
\item $\ho_i(\N_j(\Delta)) = 0$ for $i +j > d$ and $1\leq j\leq d$ (see Corollary \ref{main1}).
\item $\depth (k[\Delta]) =  \inf \{i+j \colon \tilde H_i(\N_j(\Delta)) \neq 0\}$ (see Theorem \ref{main2}).
\item For $i\geq 0$, $f_i(\Delta) = \displaystyle \sum_{j=i+1}^d \binom{j-1}{i} \chi(\N_{j}(\Delta))$ (see Theorem \ref{fvectorthm}).

\end{enumerate}
\end{theorem}

In short, the numbers $b_{ij}= \dim \ho_i(\N_j(\Delta))$ for $0\leq i\leq d-j$ and $1\leq j \leq d$ can be presented in a nice table which
determine both the depth and the $f$-vector (and thus also the $h$-vector) of $\Delta$. We provide an explicit example below.

\begin{example}
Consider the simplicial complex $\Delta$ with facets $\{ ABCD, BCDE, DEFG, DFGH \}$.

The following are geometric realizations of the complex and its higher nerves:
\captionsetup[figure]{name=Table}

\begin{figure}[h!]

\begin{tabular}{|c|c|c|c|c|} \hline
$\Delta$ & $\N_1$ & $\N_2$ & $\N_3$ & $\N_4$\\ \hline
\drawintro{1.1}&\drawintroNone{1.1}  &\drawintroNtwo{1.1} &\drawintroNthree{1.1} &\drawintroNfour{1.1}\\ \hline
\end{tabular}
\caption{Nerves of $\Delta$}
\end{figure}

\begin{figure}[h!]
\[\begin{array}{|c|c|c|c|c|c|c|} \hline
\mbox{} & \mbox{} & \mbox{} & \mbox{} & \mbox{}\\[-1em] 
\mbox{} & \tilde{H}_0 & \tilde{H}_1 & \tilde{H}_2 & \chi \\ \hline
\N_1  & 0 & 0 & 0 & 1 \\
\N_2  & 0 & 0 & 0 & 1\\
\N_3  & 1 & 0 & 0 & 2\\
\N_4  & 3 & 0 & 0 & 4\\ \hline

\end{array}\]
\caption{Nerve Homologies}
\end{figure}

\captionsetup[figure]{name=figure}

Using our main theorem and the Table 2, $ \depth k[\Delta] = 3$ and $f(\Delta) = (1, 8, 17, 14, 4 )$.

\end{example}

There are consequences to our main results.  For instance, we provide a formula to compute the regularity of any monomial ideal, not necessarily square-free, in Theorem \ref{reg}. Other algebraic properties such as Serre's condition $(S_r)$ can also be detected from the nerve table (\cite{HL17}).

\begin{remark}

Though we will not consider it in this paper, one can also define higher nerves in a more general setting.  Let $A$ be a collection of subsets of a topological space $X$. Define $N_i(A):=\{F \subseteq [r] \colon \dim \cap_{j\in F}A_{j} \geq i\}$, where $\dim$ represents Krull dimension.  In this setting, special interest is given to the case where $X$ is a Noetherian algebraic scheme; in this case, one sets $N_i(X):=N_i(A)$, where $A$ is the collection of irreducible components of $X$.  In particular, if $X=\Spec R$ for a local ring $R$, then the $N_i(X)$ provide a natural generalization of the Lyubeznik complex of $R$ (see \cite[Theorem 1.1]{Ly07} for the definition).  If, instead, $X=\Spec R$ for $R$ a Stanley-Reisner ring of a simplicial complex $\Delta$, then the complex defined in this remark coincides with that of Definition \ref{highernervedef}, via the Stanley-Reisner correspondence.  Our results in the Stanley-Reisner case raise some intriguing questions about higher nerve complexes of local schemes that can be viewed as extensions of results by Hartshorne and Katzman-Lyubeznik-Zhang (\cite{Ha62,KL16}).  (See the first version of this work, published on the arXiv: https://arxiv.org/pdf/1710.06129v1.pdf.)

\end{remark}

We now briefly describe the structure of our paper. In Section \ref{sec2}, we cover combinatorial background and fix the notation we will use throughout the paper. In Section \ref{sec3}, we recall and prove certain basic facts about depth and connectivity of a complex, which motivate our results and will be used in our proofs.  We provide a strengthened version of the classical Nerve Theorem that suits our purpose in Proposition \ref{newNerve}.  This proposition is a critical component of parts $(1)$ and $(2)$ of our main theorem.  We conclude this section by proving part $(1)$ of our main theorem. In Section \ref{lemmas}, we provide several lemmata, the main technical tools of most of our proofs. Section \ref{main} is devoted to the proof of the second part of our main theorem.  Section \ref{fvector} gives the proof of the third part of our main theorem and provides a formula for the $h$-vector in terms of homologies of higher nerves in Corollary \ref{hvector}. Section \ref{9} applies our main theorem to give a formula for computing the Castelnuovo-Mumford regularity of any monomial ideal.

\section{Notation and Definitions}\label{sec2}

In this section we introduce the notation we will use throughout this paper.  Unless otherwise stated, we fix the field $k$ and let $\tilde{H}_{i}$ denote $i$th reduced simplicial or singular homology, whichever is appropriate, always with coefficients in $k$.

We will use \(V(\Delta)\) to represent the vertex set of a simplicial complex $\Delta$; we will use $V$ instead of $V(\Delta)$ when the choice of $\Delta$ is clear; we also set $n:=|V(\Delta)|$ and $S:=k[x_1,\dots,x_{n}]$. We denote a subcomplex of $\Delta$ induced on the vertex set $W$ as $\Delta|_{W} := \set{F \in \Delta \colon F \subseteq W}$.

Given a subset $T \subseteq V(\Delta)$, we may define the star, the anti-star, and the link of $T$, denoted $\str_{\Delta}(T)$, $\astar_\Delta(T)$, and $\lk_\Delta(T)$, respectively, as follows:
\begin{align*}
\str_\Delta T &:= \set{G \in \Delta \colon T \cup G \in \Delta} \\
\astar_\Delta T &:= \set{G \in \Delta \colon T \cap G = \varnothing} = \Delta|_{V\setminus T} \\
\lk_\Delta T &:= \set{G \in \Delta \colon T \cup G \in \Delta \textrm{ and } T \cap G = \varnothing} = \str_\Delta T \cap \astar_\Delta T
\end{align*}

The star and link of $T$ are the void complex exactly when $T \notin \Delta$, and the link of $T$ is the irrelevant complex $\{\varnothing\}$ exactly when $T$ is a facet.  On the other hand, the anti-star of any $T \subsetneq V(\Delta)$ is nonempty.

We call \(\Delta^{(k)} := \{ \sigma \in \Delta : |\sigma| \le k+1 \}\) the \textit{\(k\)-skeleton of \(\Delta\)}. 

\begin{defn} \label{Fgeq} Let \(\mathcal{F}_{>k}(\Delta)\) denote the face poset of \(\Delta\) restricted to faces of $\Delta$ with cardinality strictly greater than $k$.
\end{defn}

We note the face poset of $\Delta$ is $\mathcal{F}_{>-1}(\Delta)$. Furthermore \(\mathcal{F}_{>d}(\Delta)\) is the empty poset.

\begin{defn} The \textit{order complex} of a poset $P$, denoted $\mathcal{O}(P)$, is the simplicial complex whose faces are all chains in $P$.
\end{defn}

We will denote the geometric realization of $\Delta$ as $|| \Delta ||$.

Given a complex $\Delta$, its \textit{barycentric subdivision} may be defined as $\sd \Delta := \mathcal{O}({\mathcal{F}_{>0}(\Delta)})$. The following is well-known (see Corollary 5.7 of \cite{Gi77} for example).

\begin{lem}\label{basecase}
The realization $||\Delta||$ is homeomorphic to $||\sd \Delta||$. In particular, \(\tilde{H}_i(\Delta) = \tilde{H}_i(\sd \Delta)\) for all $i$.
\end{lem}

We let $\rho : \mathcal{F}_{>0}(\Delta) \to V(\sd \Delta)$ be the map which sends an element of $\mathcal{F}_{>0}(\Delta)$ to itself viewed as a vertex of $\sd \Delta$.

We will often use the following shorthand:
\begin{align*}
[\Delta]_{>k} &= \mathcal{O}(\mathcal{F}_{>k}(\Delta)) \\
&= \sd \Delta \big|_{V(\sd \Delta) \setminus V(\sd (\Delta^{(k-1)}))}
\end{align*}

Notice that the image of $\rho$ may be restricted to $V([\Delta]_{>k})$ by restricting its domain to \(\mathcal{F}_{>k}(\Delta)\).
A \textit{simplicial map} $f:\Delta_1 \to \Delta_2$ is a function $f:V(\Delta_1) \to V(\Delta_2)$ so that for all $\sigma \in \Delta_1$, $f(\sigma) \in \Delta_2$.
We say a simplicial map $f$ is a \textit{simplicial isomorphism} if $f$ has an inverse that is a simplicial map. Note that if $f : Q \to P$ is an order-reversing or order-preserving poset map, then $f : \mathcal{O}(Q) \to \mathcal{O}(P)$ is a simplicial map.

Given a simplicial complex $\Delta$, we also consider algebraic properties of its Stanley-Reisner ring. Readers unfamiliar with the algebraic terminology used may see \cite{BH98} or a similar text for more background. Unless otherwise stated, we write $d$ for $\dim k[\Delta]$, the Krull dimension of the ring $k[\Delta]$. We also use $s(\Delta)$ to mean the minimal cardinality of facets of $\Delta$. By $\depth k[\Delta]$ we mean the depth of the $k$-algebra $k[\Delta]$; for a combinatorial characterization of $\depth k[\Delta]$, see Corollary \ref{reisner}. We say that $\Delta$ is Cohen-Macaulay whenever $k[\Delta]$ is Cohen-Macaulay, that is, whenever $\dim k[\Delta]=\depth k[\Delta].$

We further note that

\begin{align*}
\dim k[\Delta] &= \max \set{|F| \colon F \textrm{ is a facet of } \Delta}  \\
\depth k[\Delta] &= \max\{i:\Delta^{(i-1)} \mbox{ is Cohen-Macaulay}\} \le s(\Delta).
\end{align*}

\section{Preparatory Results}\label{sec3}

In this section, we begin by exploring what is known in the literature and use our construction to prove some immediate results.  Many of these results follow as a consequence of our main theorem, but their immediacy shows that our construction is a natural one.  We then prove a generalization of Borsuk's nerve theorem for simplicial complexes.

We now present Hochster's formula, which will be used throughout the paper. It relates the $i^{th}$ local cohomology module of $k[\Delta]$ supported on $\m$, denoted $H^{i}_{\m}(k[\Delta])$, to the reduced homology of links of certain faces of $\Delta$. Here $\m$ is the ideal of $k[\Delta]$ generated by the residue classes of all variables in $S$.

\begin{theorem}[Hochster \cite{BH98}]
Let $\Delta$ be a simplicial complex.  Then the Hilbert series of the local cohomology modules of $k[\Delta]$ with respect to the fine grading is given by:

\[ \Hilb_{H_{\m}^i(k[\Delta])}(t) = \sum_{T \in \Delta} \dim_k \tilde{H}_{i-|T|-1}(\lk_\Delta T) \prod_{v_j \in T} \frac{t_j^{-1}}{1-t_j^{-1}}. \]
\end{theorem}

One has $\depth k[\Delta] = \min \set{i \colon H_{\m}^i(k[\Delta]) \ne 0}$ and $\dim k[\Delta] = \max \{i \colon H_{\m}^i(k[\Delta]) \ne 0\}$, so Hochster's formula allows us to characterize depth and dimension of $k[\Delta]$ in terms of homologies of links of faces. The following is a generalization of Reisner's well known criterion for Cohen-Macaulayness.

\begin{cor}
\label{reisner}
Let $\Delta$ be a simplicial complex. Then $\depth k[\Delta] \ge t$ if and only if $\ho_{i-1}(\lk_\Delta T)=0$ for all $T \in \Delta$ with $i+|T|<t.$
\end{cor}

The following theorem, known as the Borsuk Nerve Theorem, is one of the main tools for working with the classical Nerve Complex.

\begin{theorem}[{\cite[Section 9, Corollary 2]{Bo48}}]\label{classic}
$\Delta$ and $\N_1(\Delta)$ have same homotopy type. In particular, $\ho_i(\Delta) \cong \ho_i(\N_1(\Delta))$ for all $i$.

\end{theorem}

Note if $\depth {k[\Delta]} \geq t$, then $\ho_{i-1}(\Delta)=\ho_{i-1}(\N_1(\Delta))=0$ for $i<t$ by Corollary \ref{reisner} and Corollary \ref{classic}.


Following from the definition of higher nerves, we are able to quickly derive the following results. 

\begin{lem}\label{conn1}
If $i \leq s(\Delta)$ and $\N_i(\Delta)$ is connected, then $\Delta^{(1)}$ is an $i$-connected graph.
\end{lem}
\begin{proof}
Since $\N_i(\Delta) $ is connected, there is a spanning tree of $\N_i(\Delta)^{(1)}$. Let $S$ be a set of all vertices of $\Delta$ except for at most $i-1$ of them. We have that $\N_1(\Delta|_S)$ is connected, since the facets of $\Delta|_S$ are a subset of the facets of $\Delta$, and the induced spanning tree is preserved. Since connectedness is equivalent to trivial $0^{th}$ reduced homology and $\N_1(-)$ preserves reduced homology, $\Delta|_S$ is connected. Therefore $\Delta^{(1)}$ is $i$-connected.
\end{proof}


\begin{cor}\label{cor_con}
Let $t=\depth k[\Delta]$.  Then $\Delta^{(1)}$ is a $(t-1)$-connected graph.
\end{cor}
\begin{proof}
Since $\Delta^{(t-1)}$ is Cohen-Macaulay, the facet-ridge graph of $\Delta^{(t-1)}$ is connected by \cite{Ha62};  that is, between any pair of $(t-1)$-faces of $\Delta$, there is a sequence of $(t-1)$-faces, so that each consecutive pair intersects in a $(t-2)$-face. Then for any pair of facets of $\Delta$, by choosing a $(t-1)$-face for each, and finding such a sequence between them, we construct from this a sequence of facets so that each consecutive pair intersects in a $(t-2)$-face. Therefore $N_{t-1}(\Delta)$ is connected, and the result then follows from Lemma \ref{conn1}.
\end{proof}

An easy proof of Borsuk's Nerve Theorem (Theorem \ref{classic}) uses the following result.

\begin{theorem}[\cite{Qu78}, Proposition 1.6]\label{homEquiv}
Let $f : \Delta \to \mathcal{O}(P)$ be a simplicial map. If for all $x \in P$ we have that $f^{-1}(P_{\ge x})$ is contractible, then $f$ induces a homotopy equivalence between $\Delta$ and $\mathcal{O}(P)$.
\end{theorem}

This theorem also provides a proof of our generalization of the classical Nerve Theorem.  This result is probably known to experts, but we could not find the statement we need, so we provide a proof.

\begin{prop}[Generalized Nerve Theorem]\label{newNerve}

$[\Delta]_{>j}$ is homotopy equivalent to $\N_{j+1}(\Delta)$.
\end{prop}

\begin{proof}
We use a similar approach as that of Theorem 10.6 in \cite{Bj95}.

Let $P = \mathcal{F}_{>0}(\N_{j+1}(\Delta))$ and define $f: \mathcal{F}_{>j}(\Delta) \to P$ by
$$
f(\sigma) = \set{F_i : \sigma \subseteq F_i \textrm{ facet of } \Delta}.
$$
This map is order-reversing, and it is well-defined, since $|\sigma| \geq j+1$. Therefore,  $f : \mathcal{O}(\mathcal{F}_{>j}(\Delta)) \to \mathcal{O}(P)$ is a simplicial map. For any $\tau \in P$, we have that
$$
f^{-1}(P_{\ge \tau}) = \bigcap_{F_i \in \tau} F_i,
$$
which is a face of $\Delta$ and is thus contractible. Therefore, by Theorem \ref{homEquiv}, $f$ induces a homotopy equivalence between $\mathcal{O}(\mathcal{F}_{>j}(\Delta))$ and $\mathcal{O}(P)$. Since $\mathcal{O}(P)$ is the barycentric subdivision of $\N_{j+1}(\Delta)$, Lemma \ref{basecase} says that $||\mathcal{O}(P)|| \cong ||\N_{j+1}(\Delta)||$, and therefore, $\mathcal{O}(\mathcal{F}_{>j}(\Delta))=[\Delta]_{>j}$ is homotopy equivalent to $\N_{j+1}(\Delta)$.
\end{proof}

Notice when $j=0$, we recover the classical Nerve Theorem.

We may now prove part (1) of our main theorem as a corollary.

\begin{cor}\label{main1}
For a simplicial complex $\Delta$, $\ho_i(N_j(\Delta)) = 0$ for $i+j>d$ and  $1 \le j \le d$.
\end{cor}

\begin{proof}
By Proposition \ref{newNerve}, we get

\[\ho_i(N_j(\Delta)) = \ho_i([\Delta]_{>j-1}). \]

But $[\Delta]_{>j-1}$ has dimension at most $d-j$ and the result follows.  
\end{proof}

\section{Lemmata}\label{lemmas}
In this section, we introduce several lemmata that will be integral to proving our main theorem. We refer to Section \ref{sec2} for notation.

\begin{lem}\label{isolinks}

Let \(T\) be a face of $\Delta$ and $|T|=k>0$. Then, $\lk_{[\Delta]_{>k-1}}(\rho(T)) \cong [\lk_{\Delta}(T)]_{>0}$ as simplicial complexes.  In particular, $  \ho_{i}(\lk_{[\Delta]_{>k-1}} (\rho(T))) \cong \ho_{i}(\lk_{\Delta}(T))$ for every $i$.

\end{lem}

\begin{proof}

First note that if $T$ is a facet then $\lk_{\Delta}(T)=\{\varnothing\}=[\lk_{\Delta}(T)]_{>0}$.  But, since $T$ is a facet, $\{\rho(T)\}$ must be a facet of $[\Delta]_{>k-1}$, since this is a chain of maximal length containing $\rho(T)$.  Thus $\lk_{[\Delta]_{>k-1}}(\rho(T))=\{\varnothing\}=[\lk_{\Delta}(T)]_{>0}$, and thus we have the result if $T$ is a facet.

Now, suppose $T \in \Delta$ is not a facet and define $f:V([\lk(T)]_{>0}) \to V(\lk_{[\Delta]_{>k-1}}(\rho(T)))$ by $f(\rho(\tau))=\rho(\tau \cup T)$. One can check that $f$ is a simplicial isomorphism.

Then $f$ induces a homeomorphism between the geometric realizations of $[\lk_{\Delta}(T)]_{>0}$ and \linebreak $\lk_{[\Delta]_{k-1}}(\rho(T))$, and the result follows from Lemma \ref{basecase}.  

\end{proof}

\begin{lem}\label{vertexmayer}

Suppose $b$ is a non-isolated vertex of $\Delta$.  Then there is a Mayer-Vietoris exact sequence of the form 
\[ \cdots \to \ho_{i}(\Delta) \to \ho_{i-1}(\lk_\Delta(b)) \to \ho_{i-1}(\astar_\Delta(b)) \to \ho_{i-1}(\Delta) \to \cdots \]

\end{lem}

\begin{proof}

Notice that $\str_{\Delta}(b) \cup \astar_\Delta(b) = \Delta$ and $\str(b) \cap \astar_\Delta(b)=\lk_\Delta(b)$.  Since $b$ is non-isolated, $\lk_{\Delta}(b)$ is nonempty. Thus we have a Mayer-Vietoris exact sequence in reduced homology:

\[ \cdots \to \ho_{i}(\Delta) \to \ho_{i-1}(\lk_\Delta(b)) \to \ho_{i-1}(\str_{\Delta}(b)) \oplus \ho_{i-1}(\astar_\Delta(b)) \to \ho_{i-1}(\Delta) \to \cdots \]

Since $\str_{\Delta}(T)$ is a cone, it is acyclic, and the result follows.

\end{proof}

\begin{lem}\label{claim1}
Let T be a non-trivial, non-facet face of $\Delta$ with $|T|=k$. Let $i$ be such that $\ho_{i}([\Delta]_{>k-1})=\ho_{i-1}([\Delta]_{>k-1})=0$.  Then
\[\ho_{i-1}(\lk_\Delta(T)) \cong \ho_{i-1}(\astar_{[\Delta]_{>k-1}}(\rho(T))).\]
\end{lem}

\begin{proof}

Since $T$ is not a facet, $\rho(T)$ is not an isolated vertex of $[\Delta]_{>k-1}$.  Thus, Lemma \ref{vertexmayer} gives an exact sequence

\[ \cdots \to \ho_{i}([\Delta]_{>k-1}) \to \ho_{i-1}(\lk_{[\Delta]_{>k-1}}(\rho(T))) \to \ho_{i-1}(\astar_{[\Delta]_{>k-1}}(\rho(T))) \to \ho_{i-1}([\Delta]_{>k-1}) \to \cdots \]

Since $\ho_{i}([\Delta]_{>k-1})=\ho_{i-1}([\Delta]_{>k-1})=0$, we have \(\ho_{i-1}(\astar_{[\Delta]_{>k-1}}(\rho(T))) \) \( \cong \ho_{i-1}(\lk_{[\Delta]_{>k-1}}(\rho(T)))\).

By Lemma \ref{isolinks}, we have $\ho_{i-1}(\lk_{[\Delta]_{>k-1}}(\rho(T))) \cong \ho_{i-1}(\lk_{\Delta}(T))$ which gives the result.
\end{proof}

\begin{lem}\label{independent}
Let $\Delta$ be a simplicial complex and $J \subsetneq V = V(\Delta)$ such that $\dim(\Delta|_J) = 0$. Assume that $\ho_{i-1}(\Delta) = \ho_{i}(\Delta) = 0$. Then $$\ho_{i-1}(\Delta|_{V\setminus J}) \cong \bigoplus_{x\in J} \ho_{i-1}(\Delta|_{V\setminus \{x\}}).$$
\end{lem}

\begin{proof}
We will proceed by induction on $|J|$. When $|J| = 1$, the result is immediate. Suppose the result holds for any $J$ of cardinality $k$ for some $k\geq 1$, and suppose now that $|J| = k+1$. Let $x \in J$ and $J' = J \setminus \{x\}$. Suppose $\sigma \in \Delta$. If $x \in \sigma$, then $\sigma \in \Delta|_{V\setminus J'}$; otherwise if $\sigma$ contained some $y \in J'$, then $\{x,y\} \in \Delta$, contradicting the fact that $\dim(\Delta|_J) = 0$. If $x \notin \sigma$, then $\sigma \in \Delta|_{V\setminus \{x\}}$. Therefore, $\Delta = \Delta|_{V \setminus J'} \cup \Delta|_{V\setminus \{x\}}$. Note that $\Delta|_{V \setminus J'} \cap \Delta|_{V\setminus \{x\}} = \Delta|_{V \setminus J} \neq \varnothing.$

We have the following Mayer-Vietoris sequence in reduced homology:

$$ \cdots \to \ho_{i}(\Delta) \to \ho_{i-1}(\Delta|_{V\setminus J}) \to  \ho_{i-1}(\Delta|_{V \setminus J'}) \oplus \ho_{i-1} (\Delta|_{V\setminus \{x\}}) \to \ho_{i-1}(\Delta) \to \cdots$$

Because $\ho_{i-1}(\Delta) = \ho_{i}(\Delta) = 0$, we have that $$\ho_{i-1}(\Delta|_{V\setminus J}) \cong \ho_{i-1}(\Delta|_{V \setminus J'}) \oplus \ho_{i-1} (\Delta|_{V\setminus \{x\}}).$$ By induction, $\ho_{i-1}(\Delta|_{V \setminus J'}) \cong \bigoplus_{y \in J'} \ho_{i-1} (\Delta|_{V\setminus \{y\}}).$ Therefore $$\ho_{i-1}(\Delta|_{V \setminus J}) \cong \bigoplus_{x \in J} \ho_{i-1} (\Delta|_{V\setminus \{x\}}).$$
\end{proof}

\section{Depth and higher nerves}\label{main}
\begin{theorem}\label{weaken}

For a fixed $m$, the following are equivalent:
\begin{enumerate}
\item $\ho_{i-1}(\N_{j+1}(\Delta)) = 0$ for all $i,j \geq 0$ such that $i+j < m .$
\item $\ho_{i-1}(\lk_{\Delta}(T)) = 0 $ for all $i,j \geq 0,$ $|T| = j,$ and $ i+j < m .$
\end{enumerate}
\end{theorem}

\begin{proof}
We begin the proof by showing that each condition implies $m \leq s(\Delta)$ and thus we will never need to consider the case when $T$ is a facet.  Consider the first condition: if $m > s(\Delta)$, then we may take $j= s(\Delta)-1, i=1$.  This nerve will have an isolated vertex corresponding to the facet of smallest size.  The nerve will not be connected unless that facet is the only facet.  However, if this facet is the only facet, then we contradict the first condition for $j=s(\Delta), i=0$.  Now consider the second condition: suppose $m > s(\Delta)$.  Then take $j = s(\Delta)$, $i=0$.  Then we have a contradiction when $T$ is a facet.

To prove equivalence, we will induct on $j$.  Thus, let us begin by considering the case \(j=0\).  The first set of equations is then \(\ho_{i-1}(\N_{1}(\Delta)) =0 \) for all \(i < m\).  Using Theorem \ref{classic} (1), we get that this statement is equivalent to \(\ho_{i-1}(\Delta) =0 \) for all \(i < m\).  When $j=0$, the second set of equations is in fact \(\ho_{i-1}(\Delta)=0\) for all $i < m$, since \(|T| =0\) implies $T$ is the empty set. Thus we have equivalence when $j=0$.

Now, let us take as our induction hypothesis that our theorem holds for \(j=k-1\).  Consider \(j=k < m\).  Assuming either set of equations holds, the \(j=0\) case again says that \(\ho_{i-1}(\Delta) =0 \) for all \(i < m\).  By Proposition \ref{newNerve} and the \(j=k-1\) case, either set of equations yields \(\ho_{i-1}([\Delta]_{>k})=0\) for all $i < m-(k-1)$.  Therefore, we may apply Lemma \ref{claim1} for all \(i < m-(k-1)-1=m-k\).  Thus, we have
\begin{align*}
\bigoplus_{\substack{T \in \Delta \\ |T|=k}} \ho_{i-1}(\lk_{\Delta}(T)) &\cong \bigoplus_{\substack{T \in \Delta \\ |T|=k}} \ho_{i-1}(\astar_{[\Delta]_{>k-1}}(\rho(T))).\\
\intertext{Applying Lemma \ref{independent}, we get: }
\ho_{i-1}([\Delta]_{>k}) & \cong \bigoplus_{\substack{T \in \Delta \\ |T|=k}} \ho_{i-1}(\astar_{[\Delta]_{>k-1}}(\rho(T))).\\
\intertext{And by Proposition \ref{newNerve}: }
\ho_{i-1}([\Delta]_{>k}) &\cong \ho_{i-1}(\N_{k+1}(\Delta)).
\end{align*}
Thus, we have completed the proof by induction.
\end{proof}

Combining Corollary \ref{reisner} and Theorem \ref{weaken}, we obtain the second part of our main theorem, Theorem \ref{LongConj}, restated here:
\begin{theorem}\label{main2}
For a simplicial complex $\Delta$, $\depth (k[\Delta]) =  \inf \{i+j \ | \ho_i(\N_j(\Delta)) \neq 0\}$.
\end{theorem}

\begin{remark}

Since $\depth$ is a topological property (\cite[Theorem 3.1]{Mu84}), we always have $\linebreak \depth k[\Delta]=\depth k[\sd \Delta]$ by Lemma \ref{basecase}. One can apply \cite[Proposition 2.8]{Hi91} repeatedly to show that $\depth [\Delta]_{>j} \ge \depth k[\Delta]-j$ for every $j \le d$.  In particular, by Corollary \ref{reisner}, this implies $\tilde{H}_{i}(N_{j}(\Delta))=0$ for $i<\depth k[\Delta]-j$.  Therefore, one immediately obtains $\depth k[\Delta] \le \inf \{i+j \ | \ho_i(\N_j(\Delta)) \neq 0\}$.  However, the converse to \cite[Proposition 2.8]{Hi91} does not hold, even with additional hypotheses on vanishing of homology, and therefore, these methods are incapable of establishing the reverse inequality.

\end{remark}

\section{The $f$-vector and the $h$-vector}\label{fvector}
In this section, we prove part 3 of Theorem \ref{LongConj}.  We set $\chi(\N_{j}(\Delta))$ to be the Euler characteristic of $N_j(\Delta)$ and $\tilde{\chi}(\N_{j}(\Delta))$ to be the reduced Euler characteristic of $N_j(\Delta)$. 
We use $f_i(\Delta)$ to indicate the $i^{th}$ entry in the $f$-vector of $\Delta$.

\begin{theorem}\label{fvectorthm}
Let $i \geq 0$,
\[f_i(\Delta) = \sum_{j=i+1}^d \binom{j-1}{i} \chi(\N_{j}(\Delta))\]
\end{theorem}

We note that $f_{-1}$ is always $1$.

\begin{proof}


Before we proceed, we introduce some additional notation:

Let $f_{h,k}$ be the number of $h$-faces in $\N_k(\Delta)$. We note that for any complex $\Delta$, $f_{h,k}$ is $0$ for large enough $h$ and for large enough $k$.

If a face appears in $\N_{k+1}(\Delta)$, then that face also appears in $\N_{k}(\Delta)$. We wish to count the $h$-faces of $\N_{k}(\Delta)$ which first appear in $\N_k(\Delta)$. This number is given by
\[f_{h,k}-f_{h,k+1}.\]



For a collection of facets $\rho$ let $\varphi(\rho) = \cap_{F\in \rho} F$. Note that for a given $\alpha \in \Delta$, the set of $\rho$ such that $\alpha \subseteq \varphi(\rho)$ is a Boolean lattice. Let $y,x_1,\dots,x_n$ be indeterminates, and let $x_\alpha = \prod_{i\in\alpha} x_i$. Then,


\[
\sum\limits_\rho \sum\limits_{\alpha \subseteq \varphi(\rho)} (-1)^{|\rho|}x_\alpha y^{|\alpha|}
= \sum_{\alpha \in \Delta} x_\alpha y^{|\alpha|} \sum_{\substack{\rho \\ \alpha \subseteq \varphi(\rho)}} (-1)^{|\rho|}
= 0.
\]

This is because for each $\alpha$, the set of such $\rho$ is Boolean, and therefore, $\sum\limits_{\substack{\rho \\ \alpha \subseteq \varphi(\rho)}} (-1)^{|\rho|} = 0$.

Now, setting $x_i = 1$ for all $i$ and solving for the $\rho = \emptyset$ term yields:

\[
\sum\limits_{\alpha \in \Delta} y^{|\alpha|} \hspace{3pt}=\hspace{3pt}
- \sum\limits_{\rho \neq \emptyset} (-1)^{|\rho|} \sum\limits_{\alpha \subseteq \varphi(\rho)}y^{|\alpha|}
\hspace{3pt}=\hspace{3pt} \sum\limits_{\rho \neq \emptyset} (-1)^{|\rho|-1} \sum\limits_{j=0}^{|\varphi(\rho)|} \binom{|\varphi(\rho)|}{j}y^j.
\]

Taking the $(i+1)^{st}$ coefficient of each side yields:
\begin{align*}
f_i(\Delta) &= \sum\limits_{\substack{\rho \neq \emptyset \\ |\varphi(\rho)| \geq i+1}} (-1)^{|\rho|-1}\binom{|\varphi(\rho)|}{i+1} \\
&= \sum\limits_{h=0}^\infty \sum\limits_{k = i+1}^\infty (-1)^h \binom{k}{i+1} \#\{\rho \ | \ |\rho| - 1 = h, \ \rho \in N_k(\Delta) \backslash N_{k+1}(\Delta) \} \\
&=\sum_{h=0}^\infty (-1)^h \sum_{k=i+1}^\infty \binom{k}{i+i} (f_{h,k}-f_{h,k+1}) \\
&= \sum_{h=0}^\infty (-1)^h \sum_{k=i+1}^\infty   (f_{h,k}-f_{h,k+1})  \sum_{j=i+1}^k \binom{j-1}{i} \\
&= \sum_{h=0}^\infty (-1)^h \sum_{j=i+1}^d     \sum_{k=j}^\infty \binom{j-1}{i} (f_{h,k}-f_{h,k+1})\\
&= \sum_{j=i+1}^d \binom{j-1}{i} \sum_{h=0}^\infty (-1)^h f_{h,j}\\
&= \sum_{j=i+1}^d \binom{j-1}{i} \chi(\N_{j}(\Delta)).
\end{align*}


\end{proof}

For the convenience of the reader, we have worked out the corresponding formula for the $h$-vector $(h_0=1, h_1, \dots, h_d)$ of $\Delta$.  


\begin{cor}\label{hvector}
 For $k\geq 1$ we have: \[h_k(\Delta) = (-1)^{k-1} \sum_{j\geq 1} \binom{d-j}{k-1} \tilde{\chi}(N_j(\Delta)).\]
\end{cor}

We also record the following:

\begin{cor}\label{homtop}

If $\Delta_1$ and $\Delta_2$ are simplicial complexes with $\tilde{H}_{i-1}(N_j(\Delta_1)) \cong \tilde{H}_{i-1}(N_j(\Delta_2))$ for all $i,j$, then $\Delta_1$ and $\Delta_2$ have identical $f$-vectors and h-vectors.

\end{cor}

\section{LCM-lattice and regularity of monomial ideals}\label{9}

In this section, we use our main theorem, Theorem \ref{LongConj}, to deduce a formula for the Castelnuovo-Mumford regularity of any monomial ideal $I$, denoted by $\reg(I)$. We first fix some notation motivated by \cite{GP99}. Suppose $f_1,\dots, f_r$ are the minimal monomial generators of $I$.

\begin{defn}
We define the $j$-th LCM complex of $I$ to be:

\[L_j(I) := \{ F \subseteq [r] \colon |\lcm_{i\in F}(f_i)| \leq j \}.\]

\end{defn}

\begin{theorem}\label{reg}
Let $I$ be a monomial ideal. Then:
$$\reg(I) = \sup\{j-i \ | \tilde H_i(L_{j}(I)) \neq 0\} .$$
\end{theorem}

\begin{proof}
Let $I^{pol} = (g_1,\dots,g_r)$ be the polarization of $I$. Then it is well-known that $\reg(I) = \reg(I^{pol})$ (see for instance \cite[Theorem 21.10]{PE11}). From the construction of the $g_i$'s from the $f_i$'s, it is obvious that for any subset $F \subseteq [r]$, $\lcm_{i\in F}(f_i)$ and $\lcm_{i\in F}(g_i)$ have the same size. Thus, the problem reduces to the case when $I$ is a square-free monomial ideal.

Now let $I^{\vee}$ be the Alexander dual of $I$. It is the Stanley-Reisner ideal of some complex $\Delta$. We have that $$\reg(I) = \pd S/I^{\vee} = n-\depth S/I^{\vee}$$
by the Eagon-Reiner theorem (\cite[Theorem 5.59]{MS05}) and the Auslander-Buchsbaum formula. We now note that each $g_i$ is precisely the product of variables in the complement of the corresponding facet $F_i$ of $\Delta$. Thus $L_j(I) = N_{n-j}(\Delta)$. Putting all of these together, we have:

$$\reg(I)= n-\inf \{i+j \ | \tilde H_i(L_{n-j}(I)) \neq 0\}=\sup\{j-i \ | \tilde H_i(L_{j}(I)) \neq 0\}$$

as desired.

\end{proof}

\begin{remark}
Our formula above should be compared with Theorem 2.1 in \cite{GP99}.

\end{remark}

\section*{Acknowledgments}
We thank the Mathematics Department at KU for excellent and stimulating working conditions. We are grateful to Josep \`Alvarez Montaner,  Sam Payne, Vic Reiner, Jay Schweig and Kazuma Shimomoto for helpful comments on earlier drafts of the paper. We would like to thank two anonymous referees whose suggestions greatly improved the paper. The first author is partially supported by NSA grant FED0073853. The second author would like to acknowledge that this work was partially supported by the National Science Foundation under Grant No. DMS-1440140 while the author was in residence at the Mathematical Sciences Research Institute in Berkely, California, during the Fall 2017 semester.

\bibliographystyle{amsalpha}
\bibliography{mybib}

\end{document}